\documentclass[a4paper, 12pt, reqno]{amsart}
\usepackage[utf8]{inputenc}

\usepackage{amsmath}
\usepackage{amsfonts}
\usepackage{amssymb}
\usepackage{lipsum}
\usepackage{physics}
\usepackage{esint}

\usepackage{amsthm}
\usepackage{enumerate}
\usepackage{amsrefs}

\usepackage{graphicx}
\usepackage{tikz}

\usepackage[titletoc]{appendix}
\usepackage[english]{babel}
\usepackage[utf8]{inputenc}
\usepackage[none]{hyphenat}

\usepackage{geometry}
\usepackage{hyperref}
\hypersetup{
    colorlinks=true,
    linkcolor=blue,
    citecolor=blue,
    filecolor=magenta,      
    urlcolor=cyan,
}

\newcommand\blfootnote[1]{%
  \begingroup
  \renewcommand\thefootnote{}\footnote{#1}%
  \addtocounter{footnote}{-1}%
  \endgroup
}

\newcounter{foo}

\theoremstyle{plain}
\newtheorem{thm}[foo]{Theorem}
\newtheorem{prop}[foo]{Proposition}
\newtheorem{lem}[foo]{Lemma}
\newtheorem{cor}[foo]{Corollary}

\theoremstyle{definition}
\newtheorem{defn}[foo]{Definition}

\theoremstyle{remark}
\newtheorem{rem}[foo]{Remark}

\swapnumbers

\newcommand{\p}{\partial}
\newcommand{\bp}{\overline{\partial}}

\newcommand{\Ric}{\text{Ric}}
\newcommand{\Rm}{\text{Rm}}

\newcommand{\RR}{\mathbb{R}}
\newcommand{\CC}{\mathbb{C}}
\newcommand{\PP}{\mathbb{P}}

\usepackage{hyperref}
\usepackage{cleveref}

\numberwithin{foo}{section}
\numberwithin{equation}{section}

\title{Calabi Symmetry and the Continuity Method}

\begin{document}
\date{\today}
\author{Hosea Wondo}

\maketitle

\begin{abstract}
We study the convergence and curvature blow up of La Nave and Tian's continuity method on a generalised Hirzebruch surface. We show that the Gromov-Hausdorff convergence is similar to that of the  K\"ahler-Ricci flow and obtain curvature  estimates. We also show that a general solution to the continuity method either exist or all times, or the scalar curvature blows up. This behavior is known to be exhibited by the K\"ahler-Ricci flow. 
\end{abstract}

\section{Introduction}
\blfootnote{2020 Mathematics Subject Classification. Primary 53C55; Secondary 32Q15.}
Let $(X,\omega_0)$ be a compact $n$-dimensional K\"ahler manifold and $\omega(t)$ be a deformation of $\omega_0$ given by the La Nave and Tian continuity method 
\begin{equation}\label{cont}
    \begin{cases}
        \omega(t) = \omega_0 - t \Ric(\omega(t)),\\ 
        \omega(0) = \omega_0.
    \end{cases}
\end{equation}


Metric deformation equations have become a central tool in geometric analysis in which a family of metrics deforming to a  possibly degenerate target metric is studied. This idea appears in Yau's solution to the Calabi conjecture \cite{Y78} in which a continuity method  was employed to demonstrate that one could deform an arbitrary metric to one that is Ricci flat. Another direction utilising this idea is in the  study degenerating Ricci flat metric (see \cite{T10,GTZ13,GTZ20,L21}). 

We can also consider deformations according to an evolution with time derivatives, a famous example  being the K\"ahler Ricci flow. The success of the K\"ahler-Ricci flow in solving geometric problems, has generated much interest in the K\"ahler-Ricci flow. As a mathematical tool the K\"ahler Ricci flow is hoped to give geometric classification of algebraic varieties. This classification program is known as Mori's minimal model program and aims to find a suitable simplest representative metric within every birational class. The analytic minimal model program initiated in \cite{ST12} proposes to achieve this by evolving a metric to this simplest metric through the K\"ahelr Ricci flow.  

The continuity method, given by \eqref{cont}, was introduced in \cite{R08, LT16} and is proposed as an alternative to the K\"ahler-Ricci flow in carrying out the minimal model program. This equation can be obtained by formally discretising the K\"ahler Ricci flow. The main advantage of the continuity method is that is maintains a lower Ricci bound along the deformation and therefore comparison geometry techniques, such as Cheeger–Colding–Tian’s compactness theory, can be readily applied. For recent developments pertaining to the continuity method see \cite{ZZ19,yZ19,ZZ20,FGS20,W22}. In the same way as one can generalise the K\"ahler Ricci Flow to the Chern-Ricci Flow for non-K\"ahler manifolds (see \cite{G10,TW15,LT20,TW22}), the continuity method can analogously be generalised to the Chern continuity method (see \cite{SW20,LZ21}). 

To understand geometric deformations such as the flow and continuity method, it is essential to consider several prototype examples. To this end, we consider the following $n$-dimensional K\"ahler manifold which generalises the Hirzebruch surface,

\begin{equation}\label{Xmanifold}
    X = X_{n,k} =  \PP (\mathcal{O}_{\CC \PP^{n-1}} \oplus \mathcal{O}_{\CC \PP^{n-1}}(-k)).
\end{equation}
Here we have $\mathcal{O}_{\CC \PP^{n-1}}$ to denote the trivial bundle over $\CC \PP^{n-1}$ and $\mathcal{O}_{\CC \PP^{n-1}}(-k)$ to be $k$ times tensored hyperplane bundle of $\CC \PP^{n-1}$. The manifold $X_{n,k}$ is a $\CC \PP^1$ hyperplane bundle over $\CC \PP^{n-1}$ and forms an interesting family of manifolds. For instance, the space of Cohomology classes for these manifolds, $H^{1,1}(X,\RR)$, are spanned by just two classes of exceptional divisors, $[D_0]$ and $[D_\infty]$, and therefore is a two dimensional space \cite{SW11}. As a result, the K\"ahler cone can be written explicitly as 
\begin{equation}
\mathcal{K}=\left\{-\frac{a}{k}\left[D_0\right]+\frac{b}{k}\left[D_{\infty}\right] \mid 0<a<b\right\}.
\end{equation}
It is therefore possible to obtain a concise picture of the evolution of Cohomology classes for the continuity method and for other geometric deformations such as the K\"ahler Ricci Flow. 

In \cite{LT16}, it was shown that solution to the continuity method exists until the K\"ahler class leaves the K\"ahler cone, that is up to time 
\begin{equation}\label{SingT}
T=: \sup \left\{t \mid\left[\omega_0\right]-t c_1(X)>0\right\} .
\end{equation}

If the initial metric satisfies the Calabi symmetry, then the solution $\omega(t)$ also satisfies the Calabi symmetry. We aim to study the convergence of such solutions on $X$ and the curvature blow up rate as the solution approaches this singular time. To describe this, we borrow some terminology from the study of finite time K\"ahler Ricci flows. 


\begin{defn}
Let $\omega(t)$ be a solution to the continuity method for $t \in [0,T)$. We say that the curvature tensor norm develops a \textit{Type I singularity} if 
\begin{equation}
    \sup_{X}|\Rm(\omega(t))|_{\omega(t)} \leqslant \frac{C}{T-t}. 
\end{equation}
Otherwise, we say that the curvature tensor norm blows up at a \textit{Type IIa rate}. If at a point $p \in X$, there exists a sequence $t_i \rightarrow T$ as $i \rightarrow \infty$, such that 
\begin{equation}
    |Rm(\omega(t_i))|_{\omega(t_i)} (p) \geqslant \frac{C}{T-t_t},
\end{equation}
then we say that the curvature tensor norm at $p$ blows up at an \textit{essential Type I rate}.  
\end{defn}
We can also apply the above definition to the norm of the Ricci curvature and scalar curvature. 

There has been several studies of the K\"ahler Ricci flow on $X_{n,k}$ and similar manifolds with Calabi Symmetry. The study of Gromov Hausdorff convergence of the K\"ahler Ricci flow on $X_{n,k}$ was done in \cite{SW11}. Curvature blowup rates and a blow up analysis was derived in \cite{S15}. Many arguments and calculations in this manuscript are inspired by these two papers. For a related manifold $\mathbb{P}\left(\mathcal{O}_{\mathbb{P} n} \oplus \mathcal{O}_{\mathbb{P} n}(-1)^{\oplus(m+1)}\right)$, Gromov-Hausdorff convergence was studied in \cite{SY12}.

It will also be convenient to consider the normalised equation for \eqref{cont}
\begin{equation}\label{n1cont}
    \begin{cases}
    \widetilde{\omega}(t) = \frac{1}{T-t}\widetilde{\omega}_0 - \frac{t}{T-t}\Ric(\widetilde{\omega}(t))\\ 
    \omega(0) = \omega_0.
    \end{cases}
\end{equation}

We now state the first main result of the theorem. 
\begin{thm} \label{THM1}
On the manifold $X_{n,k}$ defined by \eqref{Xmanifold}, a solution to the continuity method \eqref{cont} with an initial metric $\omega_0$ satisfying the Calabi symmetry  evolves as follows:
\begin{enumerate}
    \item If $k \geqslant n$ or $1 \leqslant k \leqslant n-1$ with $a_{0}(n+k)>b_{0}(n-k)$, then the solution exists until $T = (b_0-a_0)/2k$ and $\left(X_{n, k}, g(t)\right)$ converges to $\left(\CC \mathbb{P}^{n-1}, a_T g_{\mathrm{FS}}\right)$ as $t \rightarrow T $. The norm of the curvature tensor develops a Type I singularity. 
    \item If $1 \leqslant k \leqslant n-1$ with $a_{0}(n+k)=b_{0}(n-k)$, then the re-scaled equation \eqref{n1cont} converges to the unique K\"ahler Einstein metric $\omega_{KE}$ in the class $C_1(X)$, given such a metric exists. The norm of the curvature tensor blows up at a Type I rate. 
    \item If $1 \leqslant k \leqslant n-1$ with  $a_{0}(n+k) < b_{0}(n-k)$, then the curvature blows up at the rate
    \begin{equation}
        |\Rm(\omega)|_\omega  \leqslant \frac{C}{(T-t)^2}.
    \end{equation}
\end{enumerate}
\end{thm}
 
 
\begin{rem}
    With regards to Case II, the question of the existence of K\"ahler-Einstein metric on Fano manifolds have been extensively studied (for a survey see  \cite{T12}). 
\end{rem}

\begin{thm}\label{THM2}
    Solutions to the continuity method on $X_{n,k}$ satisfying the Calabi Symmetry has scalar curvature bound 
    \begin{equation}
        R(t) \leqslant \frac{C}{T-t}.
    \end{equation}
    Moreover, the scalar curvature blow up in Case I and II is essential in a sense that there exists $(p_i,t_i)$ such that as $i \rightarrow \infty$, $t_i \rightarrow T$ and 
    \begin{equation}
        R(p_i,t_i) \geqslant \frac{C}{T-t_i}.
    \end{equation}
\end{thm}

The behaviour of scalar curvature for the continuity method is not well understood, even for infinite time solutions to \eqref{cont}. On the other hand, we have several results regarding scalar curvature  under the K\"ahler-Ricci flow;
\begin{itemize}
    \item[(1)] The K\"ahler-Ricci flow either the flow exists for all time, or the scalar curvature blows up  \cite{Zh10}.
    \item[(2)] Solutions to the normalised K\"ahler-Ricci flow with semi-ample canonical line bundle has bounded scalar curvature \cite{Z09,ST16}. 
    \item[(3)] In the same setup as the previous point, the scalar curvature away from singular divisors converges to $-kod(X)$ \cite{J20}. 
\end{itemize}

Theorem \ref{THM2} alludes that scalar curvature must blow up for finite time singularities, resembling the result from finite time K\"ahler-Ricci flows in \cite{Zh10}. We show that solutions to the continuity method \eqref{cont} also satisfy point (1). Point (2) is currently not known globally, however a bound can be obtained on compact sets away from singular fibers \cite{ZZ19}. For point (3), Y. Zhang and Jian proved that the a twisted scalar curvature converges to $-kod(X)$. 
\begin{thm}\label{THM3}
Let $\omega(t)$ be a solution to the continuity method, \eqref{cont}. Then either the solution exists for all time, or the scalar curvature blows up; 
\begin{equation}
\sup _{X \times[0, T)} R\left(\omega(t)\right) =+\infty.
\end{equation}
\end{thm}
\begin{rem}
    The scalar curvature blowup must be from above since the Ricci curvature of solutions for \eqref{cont} is always bounded from below. 
\end{rem}

We end the introduction with an outline of the paper. In Section \ref{S2}, we describe the Calabi Symmetry condition and the geometry of $X_{n,k}$ in more detail. This will allows us to introduce the three cases for types of behaviour for solutions of \eqref{cont} at the end of this section. The next section contains three subsections where we analyse each case separately. They will follows the same format; first convergence as the solutions approach singular time established, then curvature estimates are obtained. In Section \ref{S6}, we prove Theorem \ref{THM3}. \\ 

\textbf{Acknowledgement}: I would like to thank my advisors Zhou Zhang and Haotian Wu for their encouragement and support. 

\section{Preliminaries}\label{S2}

In this section we recall the Calabi Symmetry and briefly describe the geometry of $X_{n,k}$. For a detailed construction, see \cite{Cal82} or preliminary/background sections in \cite{SW11,SY12,S15}. As stated in the introduction, the manifold
\begin{equation}
X=\mathbb{P}\left(\mathcal{O}_{\mathbb{C P}^{ n-1}} \oplus \mathcal{O}_{\mathbb{C P}^{n-1}}(-k)\right),
\end{equation}
is a $\CC \PP^1$ bundle over $\CC \PP^{n-1}$ obtained by blowing up $\CC \PP^{n-1}$. Let $D_0$ be the exceptional divisor of $X$ defined by the image of the section $(1,0)$ of $\mathcal{O}_{\mathbb{C P}^{n-1}} \oplus \mathcal{O}_{\mathbb{C P}^{n - 1}}(-k)$ and $D_{\infty}$ be the divisor of $X$ defined by the image of the section $(0,1)$ of $\mathcal{O}_{\mathbb{C P}^{n-1}} \oplus \mathcal{O}_{\mathbb{C P}^{n-1}}(-k)$. Both the 0-section $D_0$ and the $\infty$-section are complex hypersurfaces in $X$ isomorphic to $\mathbb{C P}^{n-1}$. The classes $[D_\infty]$ and $[D_0]$ span $H^{1,1}(X,\RR)$, and the K\"ahler cone is given by 
\begin{equation}
\mathcal{K}=\left\{-\frac{a}{k}\left[D_0\right]+\frac{b}{k}\left[D_{\infty}\right] \mid 0<a<b\right\}.
\end{equation}

Using the bundle map  $\pi: X_{n,k} \rightarrow \CC \PP^{n-1}$, we can pull back the Fubini-study metric on $\CC \PP^{n-1}$. Taking its cohmology class gives the  the class $[D_H]$, which is realted to $[D_0]$ and $[D_\infty]$ by 
\begin{equation}
k\left[D_{H}\right]=\left[D_{\infty}\right]-\left[D_{0}\right].
\end{equation}
We can then rewrite the K\"ahler cone in a more instructive basis  
\begin{equation}
\mathcal{K}=\left\{\frac{b-a}{k}\left[D_{\infty}\right]+a\left[D_{H}\right] \mid 0<a<b\right\}.
\end{equation}


To introduce the Calabi Symmetry into our analysis, we consider the standard holomorpic coordinates $(x_1,...,x_n)$ on $\CC^n \backslash \{0\}$, the manifold $\CC \PP^{n-1} := (\CC^n \backslash \{0\})/\CC^*$ is charted by $n$ coordinate charts $U_1,...,U_n$ where $U_i$ is characterised by $x_i \neq 0$. For fixed $i$, the holomorphic coordinates are given by $z_{(i)}^j = x_j/x_i$. In this coordinate system, the fiber coordinate $y_{(i)}$ on $\pi^{-1}(U_i)$ transforms by 
\begin{equation}
y_{(\ell)}=\left(\frac{x_{\ell}}{x_{i}}\right)^{k} y_{(i)}, \quad \text { on } \pi^{-1}\left(U_{i} \cap U_{\ell}\right),
\end{equation}
where $\pi: X_{n,k} \rightarrow \CC \PP^{n-1}$ is the bundle map.  Consider K\"ahler metrics such that $g_{i \bar{j}} = \p_i \p_{\bar{j}} u(\rho)$ for a potential $u=u(\rho)$ where 
\begin{equation}
\rho=\log \left(\sum_{i=1}^{n}\left|x_{i}\right|^{2}\right).
\end{equation}
The metrics defined in this way are invariant under the action of 
$$
G_{k} \cong U(n) / \mathbb{Z}_{k}.
$$
The condition for which $g_{ij}$ is K\"ahler is given by the following Lemma. 

\begin{lem}[Calabi \cite{Cal82}]
    A smooth convex function $u = u(\rho)$ defined on $(-\infty,\infty)$ generates a K\"ahler metric $g_{i \bar{j}} = \p_i \p_{\bar{j}} u(\rho)$ if 
    \begin{enumerate}
        \item $u'(\rho)>0$
        \item $u''(\rho) >0$
        \item There exist smooth functions $u_{[D_0]}, u_{[D_\infty]}:[0, \infty) \rightarrow \mathbb{R}$ with $u_{[D_0]}^{\prime}(0)>0$, $u_{[D_{\infty}]}^{\prime}(0)>0$ such that 
        \begin{equation}\label{DivPot1}
        u_{[D_0]}\left(e^{k \rho}\right)=u(\rho)-a \rho,
        \end{equation}
        and 
        \begin{equation}\label{DivPot2}
                \quad u_{[D_\infty]}\left(e^{-k \rho}\right)=u(\rho)-b \rho,
            \end{equation}
            for all $\rho \in \RR$.
        \end{enumerate}
\end{lem}
The divisor $D_0$ corresponds to $\rho=-\infty$ where as the $D_{\infty}$ corresponds to $\rho=\infty$. Furthermore, the coefficients $a$ and $b$ are the values of $u'(\rho)$ at the singular divisors. Indeed, taking the $\rho$ derivative of \eqref{DivPot1} and \eqref{DivPot2} gives 
\begin{equation}\label{u'ab}
    a = \lim_{\rho \rightarrow - \infty} u'(\rho) < u'(\rho) < \lim_{\rho \rightarrow \infty} u'(\rho) = b.
\end{equation}
 In the standard holomorhpic coordinates $(x_i)_{i=1}^n$ on $\CC^n \backslash \{0\}$, metrics under the Calabi condition can be written locally as  
\begin{equation}
g_{i \bar{j}}=\partial_{i} \partial_{\bar{j}} u=e^{-\rho} u^{\prime}(\rho) \delta_{i j}+e^{-2 \rho} \bar{x}_{i} x_{j}\left(u^{\prime \prime}(\rho)-u^{\prime}(\rho)\right).
\end{equation}
The determinant of the matrix $g_{i \bar{j}}$ is given by 
\begin{equation}
\operatorname{det} g=e^{-n \rho}\left(u^{\prime}(\rho)\right)^{n-1} u^{\prime \prime}(\rho).
\end{equation}

Thus, the Ricci curvature $R_{i \bar{j}} = \p_i \bp_j \log \operatorname{det} g $ is locally given by 
\begin{equation}
R_{i \bar{j}}=e^{-\rho} v^{\prime}(\rho) \delta_{i j}+e^{-2 \rho} \bar{x}_{i} x_{j}\left(v^{\prime \prime}(\rho)-v^{\prime}(\rho)\right)
\end{equation}
where 
\begin{equation}
v = - \log \det g = n \rho-(n-1) \log u^{\prime}(\rho)-\log u^{\prime \prime}(\rho).
\end{equation}

Using the expressions above, we rewrite the continuity method \eqref{cont} in terms of $u$. Denoting the $t$ variable as a subscript, $u_t(\rho)$, we obtain 
\begin{equation}\label{MAu}
    u_t(\rho) = u_0(\rho) - t (n-1) \log(u_t'(\rho)) + t \log(u_t''(\rho)) -t n\rho + c_t,
\end{equation}
where $c_t$ is chosen to be
$$ c_t = - t(n-1)\log u_t'(0) - t \log u_t''(0)$$
to ensure $u_t(0) = 0$ for all $t \in [0,T)$. Taking derivatives with respect to $\rho$ yields
\begin{equation}\label{u1}
    u_t'(\rho) = {u_0}'(\rho) + t(n-1) \frac{u_t''(\rho)}{u_t'(\rho)} + t\frac{u_t'''(\rho)}{u_t''(\rho)}-tn,
\end{equation}
and
\begin{equation}\label{u2}
    u_t''(\rho) = u_0''(\rho) + t(n-1)\frac{u_t'''(\rho)}{u_t'(\rho)} - t(n-1) \frac{u_t''(\rho)^2}{u_t'(\rho)^2} - t \frac{u_t'''(\rho)^2}{u_t''(\rho)^2}  + t \frac{u_t^{(4)}(\rho)}{u_t''(\rho)}. 
\end{equation}
For clarity of notation, we may omit the dependence of $u_t(\rho)$ on $\rho$ and write $u_t = u_t(\rho)$. We also denote the derivatives in terms of $\rho$ with $u'_t$. 

The rotational invariance simplifies many calculations since we can apply a transformation to send any point on the manifold to coordinates $x = (x_1,0...,0)$. Under this coordinate,  the evolving metric, initial metric and Ricci curvatures take the form
\begin{equation}\label{gsimpU}
\left\{g_{i \bar{j}} \right\}=e^{-\rho} \operatorname{diag}\left\{u_t^{\prime \prime}, u_t^{\prime}, \ldots, u_t^{\prime}\right\},
\end{equation}
\begin{equation}\label{g0simpU}
\left\{g^0_{i \bar{j}}\right\}=e^{-\rho} \operatorname{diag}\left\{u_0^{\prime \prime}, u_0^{\prime}, \ldots, u_0^{\prime}\right\},
\end{equation}
and
\begin{equation}\label{RicsimpU}
R_{i \bar{j}}=\sqrt{-1} e^{-\rho} \operatorname{diag}\left\{v_t^{\prime \prime}, v_t^{\prime}, \ldots, v_t^{\prime}\right\}.
\end{equation}
respectively. 
The anticanonical line bundle is given by 
\begin{equation}
K_{M}^{-1}=2\left[D_{\infty}\right]-(k-n)\left[D_{H}\right]=\frac{(k+n)}{k}\left[D_{\infty}\right]+\frac{(k-n)}{k}\left[D_{0}\right].
\end{equation}
The singular time can be deduced from the cohomology as seen by \eqref{SingT}. The maximum time of existence is therefore given by 
\begin{equation}
T=\sup \left\{t \geqslant 0 \mid \alpha_{0}+t\left[K_{X}\right]>0\right\}.
\end{equation}
For our setup, the metric's class deforms according to
\begin{equation}\label{evol}
\alpha_{t}=\frac{b_{t}}{k}\left[D_{\infty}\right]-\frac{a_{t}}{k}\left[D_{0}\right]=\frac{b_{t}-a_{t}}{k}\left[D_{\infty}\right]+a_{t}\left[D_{H}\right],
\end{equation}
where 
\begin{equation}
b_t=b_0-(k+n) t \quad \text { and } \quad a_t=a_0+(k-n) t.
\end{equation}
Finally we pick a reference for the K\"ahler class $\alpha_t$. Let 
\begin{equation}\label{uref}
\hat{u}_t(\rho)= a_t \rho+\frac{(b_t-a_t)}{k} \log \left(e^{k \rho}+1\right).
\end{equation}
Since $\p \bp \rho = \pi^* \omega_{FS} := \chi \in [D_H]$ and  $ 2 \p \bp \log \left(e^{k \rho}+1\right) \in 2[D_\infty] $, we see that the associated form $\hat{\omega}$ lies in the desired class. 
The evolution of the cohomology class exhibits three distinct behaviours.  
\begin{enumerate}
    \item \textit{Case I}: The coefficient of $[D_\infty]$ degenerates whilst $[D_H]$ remains positive. This occurs when $k \geqslant n$ or when $1 \leqslant k \leqslant n-1$ and $a_{0}(n+k)>b_{0}(n-k)$. Indeed, if the class $\alpha_t$ fails to be K\"ahler due to the coefficient of $[D_{\infty}]$ vanishing, that is  
     $$\frac{b_t-a_t}{2k} =  \frac{b_0 - a_0}{2k} -t = 0, $$
     then 
        \begin{equation}\label{IT}
        T=\frac{b_{0}-a_{0}}{2 k}.
        \end{equation}
    If $k \geqslant n$, then $a_t > a_0 >0$, but if $1 \leqslant k \leqslant n-1$ , then this uniform bound away from $0$ only holds when 
    \begin{equation}\label{0sing}
         a_T = \frac{a_0(n+k) - b_0(n-k)}{2k}>0,
    \end{equation}
    that is when with the additional condition $a_{0}(n+k)>b_{0}(n-k)$ holds. The Cohmology class therefore predicts that the continuity method contracts fibers and the resulting limiting space should be $\CC \PP^{n-1}$.
    \item \textit{Case II}: Both coefficients in $[D_\infty]$ and $[D_H]$ approach $0$ as the same singular time $T$ which occurs when $1 \leqslant k \leqslant n-1$ with $a_{0}(n+k)=b_{0}(n-k)$. As a consequence, \eqref{0sing}  yields $a_T=0$, and the K\"ahler class is proportional to the first Chern class. 
        \begin{equation}
        \alpha_t=\left(\frac{a_0}{n-k}-t\right) c_1(M).
        \end{equation}
        The singular time would then be 
        $$T =  \frac{a_0}{n-k}$$
        where we expect a full volume collapsed singularity to form. 
    \item \textit{Case III}: The coefficient of $[D_\infty]$ remains uniformly bounded from 0 and instead the coefficient of $[D_H]$ degenerates. This occurs when $1 \leqslant k \leqslant n-1$ with  $a_{0}(n+k) < b_{0}(n-k)$. The coefficient, $a_t$, of $[D_H]$  is zero when 
            \begin{equation}
                T=\frac{a_{0}}{n-k}.
            \end{equation}
    The condition for $k$ and $n$ in this case is obtained from imposing $(b_t-a_t)/2k >0$.       
\end{enumerate}

The three cases discussed above can be visualised as in Figure \ref{fig}. 

\begin{center}
\begin{figure}[ht]
\tikzset{every picture/.style={line width=0.75pt}}\label{fig} 

\begin{tikzpicture}[x=0.75pt,y=0.75pt,yscale=-1,xscale=1]

\draw  [dash pattern={on 4.5pt off 4.5pt}]  (48,47.8) -- (47,249.45) ;
\draw  [dash pattern={on 4.5pt off 4.5pt}]  (283,249.23) -- (47,249.45) ;
\draw  [dash pattern={on 4.5pt off 4.5pt}] (271,241) -- (283,249.4) -- (271,257.8) ;
\draw  [dash pattern={on 4.5pt off 4.5pt}] (39.74,58.59) -- (47.86,46.4) -- (56.53,58.21) ;
\draw    (136.5,62.83) -- (47.5,156.63) ;
\draw [shift={(88.56,113.35)}, rotate = 313.5] [fill={rgb, 255:red, 0; green, 0; blue, 0 }  ][line width=0.08]  [draw opacity=0] (8.93,-4.29) -- (0,0) -- (8.93,4.29) -- cycle    ;
\draw    (241.5,141) -- (142,248.8) ;
\draw [shift={(188.36,198.57)}, rotate = 312.71] [fill={rgb, 255:red, 0; green, 0; blue, 0 }  ][line width=0.08]  [draw opacity=0] (8.93,-4.29) -- (0,0) -- (8.93,4.29) -- cycle    ;
\draw    (187.5,100) -- (47,249.45) ;
\draw [shift={(113.83,178.37)}, rotate = 313.23] [fill={rgb, 255:red, 0; green, 0; blue, 0 }  ][line width=0.08]  [draw opacity=0] (8.93,-4.29) -- (0,0) -- (8.93,4.29) -- cycle    ;
\draw  [fill={rgb, 255:red, 0; green, 0; blue, 0 }  ,fill opacity=1 ] (178,106.5) .. controls (178,104.57) and (179.68,103) .. (181.75,103) .. controls (183.82,103) and (185.5,104.57) .. (185.5,106.5) .. controls (185.5,108.43) and (183.82,110) .. (181.75,110) .. controls (179.68,110) and (178,108.43) .. (178,106.5) -- cycle ;

\draw (295,240.4) node [anchor=north west][inner sep=0.75pt]    {$[ D_{H}]$};
\draw (29,18.4) node [anchor=north west][inner sep=0.75pt]    {$[ D_{\infty }]$};
\draw (155,78.4) node [anchor=north west][inner sep=0.75pt]    {$C_{1}( X)$};
\end{tikzpicture}

\caption[Figure 1]{Three distinct evolution of the Cohomology class of $\omega(t)$ under \eqref{cont}.}
\end{figure}
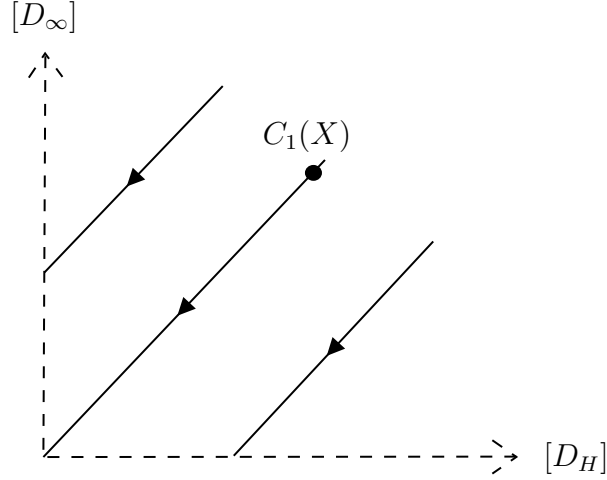
\end{center}
        
\section{Continuity Method on $X_{n,k}$}

\subsection{Analysis for Case I}\label{S3}
We  rewrite the reference metric as 
\begin{equation}\label{Iref}
\hat{\omega}_{t}=a_{t} \chi+\frac{\left(b_{t}-a_{t}\right)}{2 k} \theta=a_{t} \chi+(T-t) \theta \in \alpha_{t}.
\end{equation}
Moreover, we observe that $\theta-(k-n) \chi \in C_1(X_{n,k})$, thus we choose a form $\Omega$ such that 
\begin{equation}
\frac{\sqrt{-1}}{2 \pi} \partial \bar{\partial} \log \Omega=-\theta+(k-n) \chi.
\end{equation}
Reducing to a Monge-Ampere equation in the usual way, we obtain 
\begin{equation}\label{HMAcaseI}
    \begin{cases}
        \varphi(t) = t \log \left(\frac{\omega^n}{(T-t)\Omega}\right), \\
        \varphi(0) = 0.
    \end{cases}
\end{equation}
In Case I, the parameter $a_t$ remains uniformly bounded away from zero. Using this property, we obtain the following lemmas.
\begin{lem}\label{Ilemphi}
Let $\varphi(t)$ be a solution to \eqref{HMAcaseI}. Then there exists $C>0$ such that for all $t \in [0,T)$, 
\begin{equation}\label{Ilemphiphi}
|\varphi(t)|_{L^{\infty}(X)} \leqslant C, 
\end{equation}
and 
\begin{equation}\label{Ilemphivol}
     C^{-1} (T-t) \Omega \leqslant \omega^n \leqslant C (T-t) \Omega.
\end{equation}
\end{lem}
\begin{proof}
The form $\chi$ is a pullback of a $(n-1,n-1)$ form, hence $\chi^n=0$. Expanding out the volume form using \eqref{Iref} gives
$$\hat{\omega}_t^n = \sum_{k=0}^{n-1} a_t^k (T-t)^{n-k} \chi^k \wedge \theta^{n-k}. $$
Since $a_t$ is uniformly bounded away from 0, then for some $C>0$ independent of $t$ such that 
\begin{equation}
    C^{-1}(T-t)\Omega \leqslant \hat{\omega}_t^n \leqslant C (T-t) \Omega.
\end{equation}
Applying the maximum principle to \eqref{HMAcaseI} yields a uniform upper bound for $\varphi(t)$. The lower bound follows from the pluripotential theory of degenerate complex Monge-Amp\`{e}re equation (\cites{DP10,EGZ08}). It can be observed from \eqref{HMAcaseI} that \eqref{Ilemphiphi} and \eqref{Ilemphivol} are equivalent. 
\end{proof}
In addition, the well known Schwartz lemma holds. 
\begin{lem}
Let $\omega(t)$ be a solution to \eqref{cont}. There exists $C>0$ such that 
\begin{equation}\label{Schw}
    \operatorname{tr}_{\omega} \chi \leqslant C,
\end{equation}
for all $t \in [0,T)$. 
\end{lem}

\begin{proof}
These estimates are common and utilise a maximum principle argument. By a standard identity and rewriting any Ricci terms for $\omega = \omega(t)$ using \eqref{cont}, we obtain
\begin{equation}
    \Delta_{\omega} \log \tr_{\omega} \chi \geqslant - C_{\chi} \tr_{\omega} \chi + \frac{1}{t}\frac{\tr_{\omega_0} \chi }{\tr_{\omega}\chi} - \frac{1}{t}.
\end{equation}
Furthermore 
\begin{equation}
    \Delta_{\omega} \varphi = n - a_t \tr_{\omega}\chi + (T-t)\tr_{\omega}\theta.
\end{equation}
Combining the two lines above, we choose $A>0$ large enough so that there exists a $C = C(A)$ such that  
\begin{equation}
     \Delta_{\omega} \left( \log \tr_{\omega} \chi + A \varphi \right) \geqslant \tr_{\omega}\chi -C.
\end{equation}
Then we apply the maximum principle and since $\varphi$ is uniformly bounded by Lemma \ref{Ilemphi}, we conclude \eqref{Schw}. 
\end{proof}

\begin{lem}\label{Iulem}
The function $u_t= u_t(\rho)$ satisfying \eqref{MAu} under the range of $n,k$ in Case I satisfies:
\begin{enumerate}
    \item  The derivative $u_t'(\rho)$ remains uniformly bounded away form $0$ and is controlled from above. More precisely 
    \begin{equation}\label{Iulem1}
        a_0 < u_t^{\prime}(\rho)<a_t + 2 k(T-t)
    \end{equation}
    for all $ t \in[0, T)$. 
    \item Limiting behaviour of $u_t(\rho)$ approaching the singular time,
    \begin{equation}
        \lim _{t \rightarrow T}\left(u_t(\rho)-a_{T} \rho\right)=0. 
    \end{equation}
    \item There exist $C = C(\omega_0)$ such that 
        \begin{equation}\label{Iulem3}
        C^{-1} \frac{e^{k \rho}}{(1+e^{k\rho})^2}(T-t) <u_t^{\prime \prime}(\rho) \leqslant C \frac{e^{k \rho}}{(1+e^{k\rho})^2}(T-t)
        \end{equation}
        for all $t \in [0,T)$. 
    \item There exist $C = C(\omega_0)$ such that 
        \begin{equation}\label{Iulem4}
        \frac{\left|u_t^{\prime \prime \prime}(\rho)\right|}{u_t''(\rho)} \leqslant C
        \end{equation}
        for all $t \in [0,T)$. 
\end{enumerate}
\end{lem}

\begin{proof}
The proofs for (1), (2) and (3) can be readily adapted from \cite{SW11}, we include them here for completeness. From \eqref{u'ab} and the convexity of $u_t'$, we have 
$$ a_t \leqslant u_t' \leqslant b_t.$$
Since $k \geqslant n$, $a_t >a_0>0$, that is $u'_t$ remains uniformly bounded away from $0$. Using \eqref{IT}, we rewrite $b_t$ in terms of $a_t$;
$$ b_t = a_t + 2k(T-t),$$
to obtain the estimate in (1). To obtain (2), we recall that $u_t(0)=0$ for all $t \in [0,T)$, and so 
$$
u_t(\rho)-a_t \rho=\int_0^\rho\left(u_t^{\prime}(s)-a_t\right) d s.
$$
We then apply (i) which yields 
$$
\left|u_t(\rho)-a_t \rho\right| \leqslant 2 k(T-t)|\rho| \longrightarrow 0
$$
as $t \rightarrow T$.

For (3), we note that Lemma \ref{Ilemphi} implies
\begin{equation}
\operatorname{det} g(t) \leqslant C(T-t) \operatorname{det} \hat{g}_0,
\end{equation}
where $\hat{g}_0$ is the reference metric with associated potential $\hat{u}_0$. Thus, we estimate 
$$
\begin{aligned}
\operatorname{det} \hat{g}_0 &=e^{-n \rho}\left(\hat{u}_0^{\prime}\right)^{n-1} \hat{u}_0^{\prime \prime} \\
&=k\left(b_0-a_0\right) e^{-n \rho}\left(a_0+\left(b_0-a_0\right) \frac{e^{k \rho}}{1+e^{k \rho}}\right)^{n-1} \frac{e^{k \rho}}{\left(1+e^{k \rho}\right)^2} \\
& \leqslant C e^{-n \rho},
\end{aligned}
$$
where the second equality is obtained from \eqref{uref}. Moreover we have 
$$
\operatorname{det} g =e^{-n \rho}\left(u_t^{\prime}\right)^{n-1} u_t^{\prime \prime}.
$$
Combining the above with the previous two inequalities yield
\begin{equation}
\left(u_t^{\prime}\right)^{n-1} u_t^{\prime \prime}\leqslant C \frac{e^{k \rho}}{(1+e^{k \rho})^2}  (T-t).
\end{equation}
Since $u_t \geqslant a_t$ which is bounded uniformly away from 0, we obtain \eqref{Iulem3}. 
Similarly, using the lower bound in \eqref{Iulem3}, 
$$ 
\det(g) \geqslant C^{-1}(T-t)\det \hat{g}_0 \geqslant C^{-1}\frac{e^{k \rho}}{(1+e^{k\rho})^2}(T-t) . 
$$

To obtain the two estimates in (4), we first rewrite \eqref{u1} as 
$$ \frac{u_t'''}{u_t''} = \frac{u_t'-u_0'}{t} - (n-1) \frac{u_t''}{u_t'} +n.$$
Since $u_t'' >0$, it follows from \eqref{Iulem1} and \eqref{Iulem3} that there exists $C>0$ uniform of $t$ such that 
\begin{equation}\label{Iu'''/u''}
    \frac{|u_t'''|}{u_t''} \leqslant C. 
\end{equation}
\end{proof}

As a consequence we can recover the same results in \cite{SW11} for the continuity method. 
\begin{thm}\label{Imainthma}
We have
\begin{enumerate}
    \item $\omega(t) \geqslant a_{t} \chi.$
    \item  $\sup _{M} \operatorname{tr}_{\hat{g}_{0}} g \leqslant C$.
    \item  For any compact set $K \subset M \backslash\left(D_{\infty} \cup D_{0}\right)$,
        $$\sup _{K}\left|\nabla_{\hat{g}_{0}} g\right|_{\hat{g}_{0}} \leqslant C_{K} $$
\end{enumerate}
\end{thm}
\begin{proof}
In local coordiantes 
\begin{equation}
g_{i \bar{j}}(t)=e^{-\rho} u^{\prime} \delta_{i j}+e^{-2 \rho} \bar{x}_i x_j\left(u^{\prime \prime}-u^{\prime}\right), \quad \text{and} \quad \chi_{i \bar{j}}=e^{-\rho} \delta_{i j}-e^{-2 \rho} \bar{x}_i x_j.
\end{equation}
Then it follows easily that 
$$
g_{i \bar{j}}(t) \geqslant u_t^{\prime} e^{-\rho}\left(\delta_{i j}-\frac{\bar{x}_i x_j}{\sum_k\left|x_k\right|^2}\right) \geqslant a_t e^{-\rho}\left(\delta_{i j}-\frac{\bar{x}_i x_j}{\sum_k\left|x_k\right|^2}\right)=a_t \chi_{i \bar{j}},
$$
In local coordinates, we use \eqref{gsimpU} and \eqref{g0simpU} to compute  
\begin{equation}
\operatorname{tr}_{\hat{g}_0} g=\frac{u_t^{\prime \prime}}{\hat{u}_0^{\prime \prime}}+(n-1) \frac{u_t^{\prime}}{\hat{u}_0^{\prime}}.
\end{equation}
From Lemma \ref{Iulem}, we have 
$$ \frac{u_t'}{\hat{u}_0'} \leqslant \frac{b_0}{a_0} \quad \text{and} \quad \frac{u_t^{\prime \prime}}{\hat{u}_0^{\prime \prime}}=\frac{\left(1+e^{k \rho}\right)^2}{k\left(b_0-a_0\right) e^{k \rho}} u_t^{\prime \prime} \leqslant C, $$
which implies (2). Furthermore, 
\begin{equation}
\frac{\partial}{\partial x_k} g_{i \bar{j}}=e^{-2 \rho}\left(u_t^{\prime \prime}-u_t^{\prime}\right)\left(\bar{x}_k \delta_{i j}+\bar{x}_i \delta_{j k}\right)+e^{-3 \rho} \bar{x}_i x_j \bar{x}_k\left(u_t^{\prime \prime \prime}-3 u_t^{\prime \prime}+2 u_t^{\prime}\right).
\end{equation}
Away from the divisors $[D_{\infty}] \cup [D_{0}]$, $\rho$ and $x_i$
 are bounded which implies (3).
 \end{proof}

Following from Theorem \ref{Imainthma}, we can replicate the same argument as in \cite{SW11} verbatim to show Gromov Hausdorff convergence. 
\begin{thm}
    $(M, g(t))$ converges to $\left(\CC \mathbb{P}^{n-1}, a_T g_{F S}\right)$ in the Gromov-Hausdorff sense as $t \rightarrow T$.
\end{thm}




We now show that the norm of the curvature tensor blows up at a Type I rate. To do this, we first prove the following proposition on the scalar curvature. 
\begin{prop}\label{IR}
    For solution to \eqref{cont} in Case I, there exists $C>0$ such that 
    \begin{equation}
        R(t) \leqslant \frac{C}{T-t},
    \end{equation}
    for all $t \in [0,T).$ Furthermore, the blow up is essential everywhere on $X$. 
\end{prop}

\begin{proof}
    We first take the trace of \eqref{cont} which yields 
    $$ R(t) = \frac{1}{t}\left( \tr_{\omega} \omega_0 -n \right).$$
    In the coordinates given by \eqref{gsimpU}, the trace above becomes 
    $$R(t) = \frac{u_0''}{tu''} +(n-1)\frac{u_0'}{tu_t'} - \frac{n}{t}.$$
    Since $a_t$ is uniformly bounded from below, 
    \begin{equation}\label{Iu_t'}
        \frac{{u_0}'}{u_t'} \leqslant C,
    \end{equation}
    and using \eqref{Iulem1}, we have for compact sets $K \subset \subset X \setminus ([D_0] \cup [D_\infty])$ we have 
    \begin{equation}\label{Iu''}
        \frac{u_0''}{u_t''}  \leqslant \frac{C}{(T-t)}\frac{(1+e^{k \rho})^2u_0''}{e^{k \rho}} \leqslant \frac{C}{T-t}.
    \end{equation}
    To obtain bounds near the singular divisors, we take derivatives of both equations in \eqref{DivPot} 
    \begin{equation}\label{lim1u}
        u^{\prime \prime}_0=k^2 e^{k \rho} u_{[D_0]}^{\prime}\left(e^{k \rho}, 0\right)+k^2 e^{2 k \rho} u_{[D_0]}^{\prime \prime}\left(e^{k \rho}, 0\right),
    \end{equation} 
    and
    \begin{equation}\label{lim2u}
        u^{\prime \prime}_0=k^2 e^{-k \rho} u_{[D_\infty]}^{\prime}\left(e^{-k \rho}, 0\right)+k^2 e^{-2 k \rho} u_{[D_\infty]}^{\prime \prime}\left(e^{-k \rho}, 0\right).
    \end{equation}
    Thus, we have 
    \begin{equation}
        \lim_{\rho \rightarrow -\infty} \frac{(1+e^{k \rho})^2u_0''}{e^{k \rho}} = k^2 \quad \text{and} \quad 
        \lim_{\rho \rightarrow \infty} \frac{(1+e^{k \rho})^2u_0''}{e^{k \rho}} = k^2,
    \end{equation}
    Therefore \eqref{Iu''} holds on the entire manifold $X$ and Proposition \ref{IR} is proven. 

    Finally, we use the upper bound of $u_t''$ in Lemma  \ref{Iulem} to obtain 
    \begin{equation}\label{IRess}
        R(t) \geqslant \frac{u_0''}{tu_t''} - \frac{n}{t} > \frac{C}{(T-t)} - C,
    \end{equation}
    for any points away from singular fibers $[D_0] \cup [D_\infty]$. To extend the lower bound across the whole manifold, we compute the limits using \eqref{lim1u} and \eqref{lim2u} to obtain \eqref{IRess} on the whole of $X$. Thus, the Type I blow up rate is essential on every point in $X$. 
\end{proof}

From scalar curvature bounds, we obtain bounds on the Riemannian curvature tensor.  

\begin{prop}
Under the continuity method, we have 
\begin{equation}\label{IRbdd}
    |\Rm(\omega(t))|_{\omega(t)} \leqslant \frac{C}{T-t}. 
\end{equation}
\end{prop}

\begin{rem}
    Since we know that the scalar curvature blows up at an essential Type I rate everywhere on $X$, the curvature norm must also do the same since $R(t) \leqslant C(n) |\Rm(\omega(t))|_{\omega(t)}$ for some $C$ depending on the dimension of $X$. 
\end{rem}

\begin{proof}
 The holomorphic bisectional curvature for metrics with Calabi symmetry condition is calculated in \cite{Cao96} (also see \cite{S15}), 
\begin{equation}
\begin{aligned}
R_{i \bar{j} k \bar{l}}=& e^{-2 \rho}\left(u_t^{\prime}-u_t^{\prime \prime}\right)\left(\delta_{i j} \delta_{k l}+\delta_{i l} \delta_{k j}\right) \\
&+e^{-2 \rho}\left(3 u_t^{\prime \prime}-2 u_t^{\prime}-u_t^{\prime \prime \prime}\right)\left(\delta_{i j} \delta_{k l 1}+\delta_{i l} \delta_{k j 1}+\delta_{k l} \delta_{i j 1}+\delta_{k j} \delta_{i l 1}\right) \\
&+e^{-2 \rho}\left(6 u_t^{\prime \prime \prime}-11 u_t^{\prime \prime}-u_t^{(4)}+6 u_t^{\prime}+\frac{\left(u_t^{\prime \prime}-u_t^{\prime \prime \prime}\right)^{2}}{u_t^{\prime \prime}}\right) \delta_{i j k l 1} \\
&+e^{-2 \rho} \frac{\left(u_t^{\prime}-u_t^{\prime \prime}\right)^{2}}{u_t^{\prime}}\left(\delta_{i j \hat{1}} \delta_{k l 1}+\delta_{i l 1} \delta_{k j 1}+\delta_{k l 1} \delta_{i j 1}+\delta_{k j \hat{1}} \delta_{i l 1}\right)
\end{aligned}
\end{equation}

By applying a unitary transformation, we can send any point $p \in \CC^n \backslash \{0\}$ to $(x_1,0,...,0)$. Then all the non-vanishing term of the holomorphic bisectional curvature are given by 
\begin{equation}
\begin{aligned}
&R_{1 \overline{1} 1 \overline{1}}=e^{-2 \rho}\left(-u_t^{(4)}+\frac{\left(u_t^{\prime \prime \prime}\right)^{2}}{u_t^{\prime \prime}}\right) \\
&R_{k \bar{k} k \bar{k}}=2 e^{-2 \rho}\left(u_t^{\prime}-u_t^{\prime \prime}\right), k>1 \\
&R_{1 \overline{1} k \bar{k}}=e^{-2 \rho}\left(-u_t^{\prime \prime \prime}+\frac{\left(u_t^{\prime \prime}\right)^{2}}{u_t^{\prime}}\right), k>1 \\
&R_{k \bar{k} l \bar{l}}=e^{-2 \rho}\left(u_t^{\prime}-u_t^{\prime \prime}\right), k>1, l>1, k \neq l .
\end{aligned}
\end{equation}
In theses coordinates, by \eqref{gsimpU}, the norm of the curvature tensor is given by
\begin{equation}\label{ICurv}
    \begin{aligned}
        |\Rm(\omega)|_\omega^2 &= \sum_{k,l=1}^n g^{k\bar{k}}g^{k\bar{k}} g^{l\bar{l}}g^{l\bar{l}} R_{k\bar{k} l \bar{l}} \overline{R_{k\bar{k} l \bar{l}}} \\
        &= \left(-\frac{u_t^{(4)}}{u_t''^2}+\frac{\left(u_t^{\prime \prime \prime}\right)^{2}}{u_t''^3}\right)^2 + (n-1)(n+3) \left( \frac{1}{u_t'} - \frac{u_t''}{u_t'^2}\right)^2 + (n-1) \left( - \frac{u_t'''}{u_t' u_t''} + \frac{u_t''}{u_t'^2}\right)^2.
    \end{aligned}
\end{equation}

On the other hand, taking the trace of the Ricci curvature using  \eqref{gsimpU} and \eqref{RicsimpU} yields
\begin{equation}\label{Rvu}
    R(t) = -\frac{2(n-1)}{u'} \frac{u_t'''}{u_t''} + \frac{n(n-1)}{u_t'} - \frac{(n-1)(n-2)u_t''}{u_t'^2} + \frac{u_t'''^2}{u_t''^3} - \frac{u_t^{(4)}}{u_t''^2} 
\end{equation}

Using Lemma \ref{IR} and the estimates in Lemma \ref{Iulem}, we obtain 
\begin{equation}
    -C_1 \leqslant \frac{u_t'''^2}{u_t''^3} - \frac{u_t^{(4)}}{u_t''^2} \leqslant \frac{C_1}{T-t}. 
\end{equation}

Once again, we use the estimates in Lemma \ref{Iulem}, there exists constants $C_i>0$, $i=2,3$ such that 
$$  -C_2 \leqslant \frac{1}{u_t'} - \frac{u_t''}{u_t'^2}  \leqslant C_2,$$
and
$$ - C_3 \leqslant - \frac{u_t'''}{u_t' u_t''} + \frac{u_t''}{u_t'^2 } \leqslant C_3. $$
Combining the three inequalities with \eqref{ICurv} give us \eqref{IRbdd}.
\end{proof}

\subsection{Analysis for Case II}\label{S4}
Since the continuity method evolves cohomology as
\begin{equation}\label{IIalpha}
\alpha_t =\left (T-t \right) c_1(X),
\end{equation}
where
$$T =  \frac{a_0}{n-k}.$$

We apply a rescaling to obtain a normalised equation which keeps the class of the metric constant. From \eqref{IIalpha} we see that the appropriate rescaling is 
\begin{equation}
    \widetilde{\omega}(t) = \frac{\omega(t)}{T-t},
\end{equation}
from which \eqref{cont} becomes 
\begin{equation}\label{ncont}
    \widetilde{\omega}(t) = \frac{\widetilde{\omega_0}}{T-t} - \frac{t}{T-t}\Ric(\widetilde{\omega}(t)).
\end{equation}
This corresponds to the classical continuity method proposed by Aubin (see \cite{A76,A84,ZZ20}). If the manifold admits a unique K\"ahler Einstein metric, $\omega_{KE}$ then $\omega(t)$ converges in the $C^\infty(X,\omega_0)$ topology as $t \rightarrow T$. 

Returning to \eqref{cont}, we take the refernce metric based on Cohomology to be  
\begin{equation}\label{IIref}    \hat{\omega}_t =(T-t)((n-k)\chi + \theta). 
\end{equation}
Thus, we expect a full collapsed deformation. Like in Case I , we reduce the continuity method to a Monge-ampere equation
\begin{equation}\label{IIMA}
    \varphi(t) = t \log \left(\frac{\omega^n}{(T-t)^n\Omega}\right) \quad \text{with } \varphi(0) =0.
\end{equation}
\begin{lem}
The solution, $\varphi(t)$, the \eqref{IIMA} satisfies 
\begin{equation}
    |\varphi(t)|_{L^\infty(X)} \leqslant C
\end{equation}
for some $C>0$. Furthermore   
\begin{equation}\label{IIvol}
    C^{-1}(T-t)^n \leqslant \omega^n \leqslant C (T-t)^n \Omega.
\end{equation}
\end{lem}
\begin{proof}
The proof is similar to that in Case I, using that \eqref{IIref} implies
$$ \hat{\omega}_t^n \sim C (T-t)^n \Omega. $$
Then applying a maximum principle argument yields an upper bound. From the results in \cites{DP10,EGZ08}, we obtain a uniform lower bound. 
\end{proof}

We aim to replicate Theorem \ref{Iulem} for Case II. However, we need to make some modifications. 
\begin{lem}\label{IIulem}
The estimates for $u= u_t(\rho)$ in Lemma \ref{Iulem} hold with the following modifications to (1) and (2); 
\begin{enumerate}
    \item  Equivalence of $u_t'(\rho)$ with $a_t$, that is 
    \begin{equation}\label{IIulem1}
        (n-k)(T-t) < u_t^{\prime}(\rho)< (n+k)(T-t),
            \end{equation}
            for all $t \in [0,T)$. 
    \item Limiting behaviour of $u_t(\rho)$ approaching the singular time,
    $$ \lim _{t \rightarrow T} u_t(\rho)=0. $$
    \item As with Case I, there exists $C>0$ such that  
    \begin{equation}\label{IIulem2}
        C^{-1}\frac{e^{k \rho}}{(1+e^{k\rho})^2} (T-t)  < u_t^{\prime \prime}(\rho) \leqslant C \frac{e^{k \rho}}{(1+e^{k\rho})^2} (T-t),
    \end{equation}
    and 
    \begin{equation}
        \frac{\left|u_t^{\prime \prime \prime}(\rho)\right|}{u_t''(\rho)} \leqslant C.
    \end{equation}
\end{enumerate}
\end{lem}

\begin{rem}
    A key observation here is that $u_t$ is no longer uniformly bounded from below away from $0$.  
\end{rem}
\begin{proof}
    The proof only requires slight modification form the proof of Lemma \ref{Iulem}. Part (1) simplify follows from $a_t <  u_t' <b_t$,since $a_t = (n-k)(T-t)$ and similarly $b_t = (n+k)(T-t)$. Then for any $\rho \in \RR$
    $$
    |u_t(\rho)| = \left| \int_0^\rho u_t^{\prime}(s) ds \right|  \leqslant (T-t) |\rho| \rightarrow 0,
    $$
    as $t \rightarrow T$. 
    To obtain \eqref{IIulem2}, we again note that \eqref{Iulem3} implies 
    $$C^{-1}(T-t)^n\det(\hat{g}_0) \leqslant \det(g) \leqslant C (T-t)^n \det(\hat{g}_0).$$  
    In terms of the potential, $u_t$, the above inequality becomes 
    $$
    C^{-1}(T-t)^n \frac{e^{k \rho}}{(1+e^{k\rho})^2} \leqslant (u_t')^{n-1}u_t'' \leqslant C \frac{e^{k \rho}}{(1+e^{k\rho})^2} (T-t)^n.
    $$
    Using that $u_t \geqslant (n-k)(T-t)$ from \eqref{IIulem1},  our desired bound for $u_t''$ easily follows. The third order estimates only rely on an upper bound for $u_t''$ and thus can be derived in the same way as before in Lemma \ref{Iulem}. 
\end{proof}

With a lower bound for $u_t''$ identical to Case I, we obtain the same Type I blow up for scalar curvature and consequently 
\begin{prop}
    In Case II, the Riemannain curvature tensor blows up at a Type I rate. 
\end{prop}

\begin{proof}
    Like in Case I, we first bound scalar curvature. This can be done in the same way replacing \eqref{IRbdd} with $u_0/u_t' \leqslant C(T-t)^{-1}$. The bound \eqref{Iu''} remains unchanged along with the limit calculations for $\rho \rightarrow \pm  \infty$. Thus scalar curvature blows up at a Type I rate and \eqref{Rvu} holds. Furthermore, $u_t'$ and $u''_t < C(T-t)$ away from singular sets. Thus, the Type I blow up for scalar curvature is essential everywhere. Returning to the estimate for the curvature tensor, 
we use Lemma \ref{IIulem} to bound terms in \eqref{ICurv} as follows. There exists $C_i>0$, $i=1,2,3$ uniformly of $t$ such that 
$$ -\frac{C_1}{T-t} \leqslant -\frac{u_t^{(4)}}{u_t''^2}+\frac{\left(u_t^{\prime \prime \prime}\right)^{2}}{u_t''^3} \leqslant \frac{C_1}{T-t},$$
$$ -\frac{C_2}{T-t} \leqslant \frac{1}{u_t'} - \frac{u_t''}{u_t'^2}  \leqslant \frac{C_2}{T-t},$$
and
$$ -\frac{C_3}{T-t}  \leqslant - \frac{u_t'''}{u_t' u_t''} + \frac{u_t''}{u_t'^2 } \leqslant \frac{C_3}{T-t}. $$
Substituting the above estimate into \eqref{ICurv} reveals a type I rate. 
    
\end{proof}

\subsection{Analysis for Case III}\label{S5}
In Case III, the singular time is 
\begin{equation}
    T = \frac{a_0}{n-k}.
\end{equation}
Here $a_t = (n-k)(T-t)$ and $b_t>0$. The reference metric is given by 
\begin{equation}
    \hat{\omega}_t = \omega_0 + t((k-n) \chi-\theta) := \omega_0 + t\eta.
\end{equation}
Since $-\eta \in C_1(X)$, there exists a volume form $\Omega$ such that $\frac{\sqrt{-1}}{2 \pi} \partial \bar{\partial} \log \Omega=\eta$. One then reduce the continuity to the following Monge-ampere equations
\begin{equation}\label{IIIMA}
    \varphi(t) = t \log\left( \frac{\omega^n}{\Omega}\right).
\end{equation}
For this equation, we also have 
\begin{lem}
The solution, $\varphi(t)$ to \eqref{IIIMA} satisfies 
\begin{equation}
    |\varphi(t)|_{L^\infty(X)} \leqslant C
\end{equation}
for some $C>0$, and hence,  
\begin{equation}\label{IIIphi}
    C^{-1} \Omega \leqslant \omega^n \leqslant C \Omega.
\end{equation}
\end{lem}

The proof is identical to Lemma \ref{Ilemphi}. For the estimates on the potential $u_t(\rho)$, we prove the following Lemma. 

\begin{lem}\label{IIIulem}
The function $u= u_t(\rho)$ satisfies, for all $\rho \in \RR$,  
\begin{enumerate}
    \item  Bound on $u_t'(\rho)$,
    \begin{equation}\label{IIIulem1}
        (n-k)(T-t) < u_t^{\prime}(\rho)< (n+k)(T-t) + b_0, 
    \end{equation}
    for all $t\in [0,T)$. 
    \item Estimates on $u_t''(\rho)$. For all $\rho \in \RR$, there exists $C>0$ such that 
        \begin{equation}\label{IIIulem3}
        C^{-1}\frac{e^{k\rho}}{(1+e^{k\rho})^2} \leqslant u_t^{\prime \prime}(\rho) \leqslant \frac{Cu_t'(\rho)}{T-t},
        \end{equation}
    \item There exist $C = C(\omega_0)$ such that 
        \begin{equation}\label{IIIulem4}
        \frac{\left|u_t^{\prime \prime \prime}(\rho)\right|}{u_t''(\rho)} \leqslant \frac{C}{T-t}.
        \end{equation}
\end{enumerate}
\end{lem}

\begin{proof}
The estimate in \eqref{IIIulem1} follows from the definition of $a_t$ and $b_t$. The lower bound of \eqref{IIIulem3} can be derived from the volume estimate \eqref{IIIphi}. Similarly to the previous cases, 
$$
    C^{-1} \frac{e^{k \rho}}{(1+e^{k\rho})^2} \leqslant (u_t')^{n-1}u_t''
$$
Since in Case III, $u_t' \leqslant C$, the lower bound for \eqref{IIIulem3} follows. The same method however does not seem to produce effective bounds as in Case II since the volume is non collapsed. Hence the lower bound for $u_t'$ produces an upper bound for $u_t''$ of $(T-t)^{-(n-1)}$. To show the upper bound to part \eqref{IIIulem3}, consider $H_t(\rho) := u_t''(\rho) / u_t'(\rho)$. Using \eqref{DivPot1} and \eqref{DivPot2}, we note that 

\begin{equation}
    \lim_{\rho \rightarrow - \infty} H = \lim_{\rho \rightarrow - \infty} \frac{k^2e^{k \rho}u_{[D_0]}'(e^{k \rho})+k^2e^{2k \rho}u_{[D_0]}''(e^{k \rho})}{k^2e^{k \rho}u_{[D_0]}'(e^{k \rho}) + a_t} =0,
\end{equation}
\begin{equation}
    \lim_{\rho \rightarrow \infty} H = \lim_{\rho \rightarrow \infty} \frac{k^2e^{-k \rho}u_{[D_0]}'(e^{-k \rho})+k^2e^{-2k \rho}u_{[D_0]}''(e^{-k \rho})}{k^2e^{-k \rho}u_{[D_0]}'(e^{-k \rho}) + b_t} =0.
\end{equation}
Thus a maximum exists away from singular divisors. The derivatives of $H$ are given by
$$ H' = \frac{u_t'''}{u_t'} - \frac{u_t''^2}{u_t'^2}$$
and 
$$H'' = \frac{u_t^{(4)}}{u_t'} - \frac{ u_t'' u_t'''}{u_t'^2} + 2 \frac{u_t''^3}{u_t'^3}. $$
Using \eqref{u2} and the expression for $H'$, we rewrite the previous equality as 
\begin{equation}
    H'' = \frac{u_t'}{t}H^2 - \frac{u_0''}{t}H + \frac{(H'+H^2)^2}{H} - (n+2)H(H'+H^2)+(n+1)H^3.
\end{equation}
At maximum points $p_0$, the above equation yields 
$$ 0 \geqslant  \frac{u_t'}{t}H(p_0) - \frac{u_0''}{t}.$$
Using \eqref{IIIulem1}, we obtain 
\begin{equation}
    H(p_0) \leqslant \frac{u_0''}{u_t'} \leqslant \frac{C}{T-t}
\end{equation}
for some $C>0$ independent of $t \in [0,T)$. Since $u_t'$ is uniformly bounded, we obtain the upper bound in \eqref{IIIulem3}.
\end{proof}

\begin{prop}
The curvature tensor in Case III satisfies
\begin{equation}
    |\Rm(\omega)|_{\omega} \leqslant \frac{C}{(T-t)^2}.
\end{equation}
\end{prop}
\begin{proof}
The scalar curvature is bounded from above by
    $$R(t) = \frac{u_0''}{tu_t''} +(n-1)\frac{u_0'}{tu_t'} - \frac{n}{t} \leqslant \frac{C}{T-t}, $$
and we always have a lower scalar curvature bound. From \eqref{Rvu}, and using Lemma \ref{IIIulem}, we obtain
\begin{equation}
    -\frac{C_1}{(T-t)^2} \leqslant \frac{u_t'''^2}{u_t''^3} - \frac{u_t^{(4)}}{u_t''^2} \leqslant \frac{C_1}{(T-t)^2}. 
\end{equation}
The $u_t$ estimates also give us 
$$ -\frac{C_2}{T-t} \leqslant \frac{1}{u_t'} - \frac{u_t''}{u_t'^2}  \leqslant \frac{C_2}{T-t},$$
and
$$ -\frac{C_3}{(T-t)^2}  \leqslant - \frac{u_t'''}{u_t' u_t''} + \frac{u_t''}{u_t'^2 } \leqslant \frac{C_3}{(T-t)^2}. $$
Thus, the worst blow up term in the curvature tensor is of order $(T-t)^{-2}$.
\end{proof}

\section{Finite Time Singularity}\label{S6}
In this section, we prove Theorem \ref{THM3}; if the continuity method encounters a finite time singularity, the scalar curvature must necessarily blow up. We observe this behaviour in the analysis of Case I and II in the previous section. Furthermore, this result holds for the K\"ahler Ricci Flow (see \cite{Zh10}). 

We need some setup before proving Theorem \ref{THM3}. For convenience, we rescale \eqref{cont} by 
\begin{equation}
         (1+s) \omega(s) = \omega(t) \quad \text{with} \quad s = t. 
\end{equation}
The continuity method becomes
\begin{equation}\label{ncont}
    \begin{cases}
    (1+t) \omega(t) = \omega_0 - t \Ric(\omega(t)),\\
    \omega(0) = \omega_0.
    \end{cases}
\end{equation}
Choose a form $\omega_\infty$ in the canonical class $K_X$. By Yau's theorem there exists a $\Omega$ such that 
\begin{equation}
\omega_{\infty}=-\operatorname{Ric}(\Omega). 
\end{equation}

This allows us to define the reference metric, 
\begin{equation}\label{refmet1}
\omega_t := \frac{1}{1+t}\omega_0 + \frac{t}{1+t} \omega_\infty,
\end{equation}
which is obtained by formally solving the cohomology evolution of \eqref{ncont}. Then \eqref{cont} reduces to the scalar equation
\begin{equation}\label{scalcont1}
\begin{cases}
        \frac{1+t}{t} \varphi(t) = \log \left( \frac{(1+t)^r\omega^n}{\Omega} \right)\\
        \varphi(0) = 0.
\end{cases}
\end{equation}
The characterisation \eqref{SingT} of the maximum time of existence  still applies here and a solution to the Monge-Ampere equation above generates a solution to the continuity method. 

\begin{lem}\label{c0}
There exists a constant $C>0$ such that for any time $t \geqslant 1$ 
$$ \varphi(t) \leqslant C.$$ 
\end{lem}

\begin{proof}
Expanding out $\omega_t^n$, we have 
\begin{equation}
    \omega_t^n = \sum_{k=r}^n \binom{n}{k} \frac{t^{n-k}}{(1+t)^n} \omega_0^k \wedge \omega_\infty^{n-k} \leqslant C. 
\end{equation}
Here we have used that $t < T < \infty$. Applying the maximum principle and substituting the above into \eqref{scalcont1}, we obtain some $C>0$ such that 
$$\varphi(t) \leqslant C,$$
for all $t \in [0,T)$. 
\end{proof}
From \eqref{scalcont1}, we immediately have 
\begin{cor}\label{cor}
There exists $C>0$ such that for all $t \in [0,T)$,
\begin{equation}\label{volequiv1}
    \omega(t)^n \leqslant C (1+t)^{-r} \Omega. 
\end{equation}
\end{cor}

\begin{proof}[Proof of Theorem \ref{THM3}]

Suppose that $T< \infty$ is the maximal time given by \eqref{SingT} and assume by contradiction that $R(t) < C$ for some $C>0$ independent of $t$. From taking the trace of \eqref{cont}, we have 
$$ R = \frac{1}{t}\tr_{\omega} \omega_0 - n \left( \frac{1+t}{t}\right),$$
which implies 
\begin{equation}\label{tr1}
     \tr_{\omega} \omega_0 \leqslant C.
\end{equation}
To obtain a lower estimate for $\omega(t)$, we use the well known eigenvalue estimate and \eqref{volequiv1} to obtain a $C>0$ such that 
\begin{equation}\label{tr2}
    \tr_{\omega_0} \omega \leqslant (n-1)! (\tr_{\omega} \omega_0)^{n-1} \frac{\omega^n}{\omega_0^n} \leqslant C \frac{t^{n-1}}{(1+t)^r} \frac{\Omega}{\omega_0^n} \leqslant C,
\end{equation}
for all $t \in [0,T)$. The last inequality holds from our assumption $T< \infty$. Combining \eqref{tr1} and \eqref{tr2}, we conclude that the metric equivalence 
\begin{equation}\label{bad}
    C^{-1} \omega_0 \leqslant \omega(t) \leqslant C \omega_0.
\end{equation}
 To obtain the contradiction, we recall Demalli and  Paun's characterisation of K\"ahler metrics (Theorem 4.2 in \cite{D04})
 \begin{thm}
Let $(X, \omega_0)$ be a compact K\"ahler manifold and let $[\alpha]$ be a $(1,1)$ cohomology class in $H^{1,1}(X, \mathbb{R})$. Then $[\alpha]$ is K\"ahler if and only if for every irreducible analytic set $Y \subset X, \operatorname{dim} Y=n$, and every $t \geqslant 0$,
\begin{equation}\label{Dchar}
    \int_Y(\alpha+t \omega_0)^n>0. 
\end{equation}
 \end{thm}


However, the lower bound in \eqref{bad} reveals that \eqref{Dchar} holds for the classes $[\omega(t)]$ for all $t \in [0,T)$. From taking the limit, we see that the class $[\omega_T]$ also satisfies \eqref{Dchar} and thus, the limiting class is K\"ahler which contradicts the maximal time of existence.
\end{proof}


\bibliography{main}

\end{document}